\documentclass[a4paper,11 pt]{amsart}

\RequirePackage{amsmath, amssymb, amsthm, amsfonts}
\usepackage{fullpage}
\newtheorem{theo}{Theorem}[section]
\newtheorem{lem}[theo]{Lemma}
\newtheorem{prop}[theo]{Proposition}

\newtheorem{defin}[theo]{Definition}

\theoremstyle{definition}

\newtheorem{rem}[theo]{Remark}

\newcommand{\PP}{\mathbb{P}}

\newcommand{\oo}{\mathcal{O}}

\title{On asymptotic bounds for the number of\\ irreducible components of the moduli space of\\ surfaces of general type}
\author{Michael L\"onne and Matteo Penegini}
\address{Michael L\"onne\\ Institut f\"ur Algebraische Geometrie,  Leibniz Universit\"at Hannover, Welfengarten 1\\ D-30167 Hannover, Germany}
\email{loenne@math.uni-hannover.de}
\address{Matteo Penegini\\
Dipartimento di Matematica \emph{``Federigo Enriques''}, Universit\`{a} degli Studi di Milano, Via Saldini 50, I-20133 Milano, Italy} \email{matteo.penegini@unimi.it}

\begin{document}


\maketitle


\begin{abstract}
In this paper we investigate the asymptotic growth of the number of irreducible and connected components of the moduli space of
surfaces of general type corresponding to certain families of
surfaces isogenous to a higher product. We obtain a higher growth then the previous growth by Manetti \cite{Man}.
\end{abstract}

\section{Introduction}
 
It is known that, once two positive integers $(y,x)$ are fixed, the number of irreducible components $\iota(x,y)$ of the moduli space of surfaces of general type with $K^2=y$ and $\chi=x$ is bounded from above by a function of $y$.
In fact, Catanese proved that the number $\iota^0(y,x)$ of components containing regular surfaces, ie.\ $q(S)=0$, has
an exponential upper bound in $K^2$, more precisely \cite[p.592]{Cat92} gives the following inequality
\[
\iota^0(x,y) \leq y^{77y^2}. 
\]

There are also some results showing how close one can get to this bound from below. In
\cite{Man}, for example, Manetti constructed a sequence $S_n$ of
simply connected surfaces of general type with $K^2_{S_n}=:y_n$, such that the lower
bound for the number $\delta(S_n)$ of $\mathcal{C}^{\infty}$ inequivalent complex structures on the oriented topological $4$-manifold underlying $S_n$  is
\[
\delta(S_n) \geq y_n^{\frac{1}{5}\ln y_n}.
\]

Using group theoretical methods, we are able to describe the asymptotic growth of the number of irreducible and connected components of the moduli space of surfaces
of general type in certain sequences of surfaces. More
precisely, we apply the definition and some properties of
surfaces isogenous to a product of curves and we reduce the geometric problem of finding connected components into the algebraic one of counting
some subfamilies of $2$-groups, which can be effectively computed. Similar methods were first applied by Garion and the second author in \cite{GP11}, see also \cite{P13}. Our main result is the following.

\begin{theo}\label{thm main}\sloppy Let $h$ be number of connected components containing surfaces isogenous to a product of curves of irregularity $q(S)=q \geq 0$, admitting a group of order $2^{3s}$ and ramification structure of type $((0 | 2^{2s+2}),(q | 2^{2s-2q+2}))$. Then for $s \rightarrow \infty$ we have  
\[ h \geq 2^{\frac{2}{9}\big(\ln x_s\big)^3}.
\]
In particular, we obtain sequences $y_s$ and $x_s=y_s/8$ with
\[\iota^0(x_s,y_s) \geq y_s^{\frac{2}{13}(\ln y_s)^2}.
\]
\end{theo}
\noindent
Let us explain now the way in which this paper is organized.

The next section \emph{Preliminaries} is divided into three parts. In the first part we recall different moduli spaces of surfaces of general type that one can consider and how the number of their irreducible and connected components are related. In the second part we recall the definition and the properties of surfaces isogenous to a higher product and the its associated group theoretical data, the so called \emph{ramification structures}. The third part is purely group theoretical and we recall some generating properties of nilpotent groups of Frattini-class $2$. In particular we give the asymptotic growth of the number of certain subfamilies of $2$-groups.

In the two successive sections we construct infinitely families of regular (respectively irregular) surfaces isogenous to a product associated to nilpotent groups.

In the last section we prove the main theorem.
\bigskip

\textbf{Acknowledgement} The authors acknowledge the travel grant from a DAAD-VIGONI program. They also thank the Hausdorff Center (HIM) in Bonn for the kind hospitality. 

\bigskip

\textbf{Notation and conventions.} We work over the field
$\mathbb{C}$ of complex numbers. By ``\emph{surface}'' we mean a projective, non-singular surface $S$. For such a surface $\omega_S=\oo_S(K_S)$ denotes the canonical
class, $p_g(S)=h^0(S, \, \omega_S)$ is the \emph{geometric genus},
$q(S)=h^1(S, \, \omega_S)$ is the \emph{irregularity} and
$\chi(\mathcal{O}_S)=\chi(S)=1-q(S)+p_g(S)$ is the \emph{Euler-Poincar\'e
characteristic}.

\section{Preliminaries}
\subsection{The moduli space of surfaces of general type}
It is well known (see \cite{Gie77}) that once are fixed two positive integers $x,y$ there exists a quasiprojective coarse moduli space $\mathcal{M}_{y,x}$ of canonical models of surfaces of general type with $x=\chi(S)$ and $y=K^2_S$. We study the number $\iota(x,y)$ ($\gamma(x,y)$) of irreducible (resp. connected) components of  $\mathcal{M}_{y,x}$.

In addition, one can consider different structures on a surface of general type $S$, for example we can denote by $S^{top}$ the oriented topological $4$-manifold underlying $S$, or with $S^{diff}$ the oriented $\mathcal{C}^{\infty}$ manifold underlying $S$, and we can attach to $S$ several integers.

Let $\mathcal{M}^{top}(S)$  be the subspace of $\mathcal{M}_{y,x}$ corresponding to surfaces (orientedly) homeomorphic to $S$, and $\mathcal{M}^{{\it diff}}(S)$ be the subspace 
corresponding to surfaces diffeomorphic to $S$, we define:
\begin{description}
\item[$\delta(S)$] number of $\mathcal{C}^{\infty}$ inequivalent complex structures on $S^{top}$.
\item[$\gamma(S)$] number of connected components of $\mathcal{M}^{top}(S)$.
\item[$\iota(S)$] number of irreducible components of $\mathcal{M}^{top}(S)$.
\item[$\gamma(x,y)$] number of connected components of $\mathcal{M}_{y,x}$.
\item[$\iota(x,y)$] number of irreducible components of $\mathcal{M}_{y,x}$.
\item[$\iota^0(x,y)$] number of irreducible components of $\mathcal{M}_{y,x}$ of regular surfaces.
\end{description}

There are inequalities among the numbers above, the ones we need are the following:
\begin{equation}
\delta(S) \leq \gamma(S) \leq \iota(S) \leq \iota(x,y), \quad \gamma(x,y) \leq \iota(x,y).
\end{equation}
See also \cite{Cat92}. The union $\mathcal{M}$ over all admissible pairs of invariants ($y,x$) of
these spaces is called the \emph{moduli space of surfaces of
general type}.
\subsection{Surfaces isogenous to a product and their moduli}
\begin{defin}\label{def.isogenous} A surface $S$ is said to be \emph{isogenous to a higher product of curves}\index{Surface isogenous to a higher product of curves}
 if and only if, equivalently, either:
\begin{enumerate}
\item $S$ admits a finite unramified covering which is isomorphic
to a product of curves of genera at least two;
\item $S$ is a
quotient $(C_1 \times C_2)/G$, where $C_1$ and $C_2$ are curves of
genus at least two, and $G$ is a finite group acting freely on
$C_1 \times C_2$.
\end{enumerate}
\end{defin}
By Proposition 3.11 of \cite{cat00} the two properties
(1) and (2) are equivalent. Using the same notation as in Definition \ref{def.isogenous}, let
$S$ be a surface isogenous to a product, and $G^{\circ}:=G
\cap(Aut(C_1) \times Aut(C_2))$. Then $G^{\circ}$ acts on the two
factors $C_1$, $C_2$ and diagonally on the product $C_1 \times
C_2$. If $G^{\circ}$ acts faithfully on both curves, we say that
$S= (C_1 \times C_2)/G$ is a \emph{minimal
realization}. In \cite{cat00}
it is also proven that any
surface isogenous to a product
admits a unique minimal realization. 

\medskip

{\bf Assumptions I:} In the following we always assume:
\begin{enumerate}
\item Any surface $S$ isogenous to a product is given by its unique minimal realization;
\item $G^{\circ}=G$, this case is also known as \emph{unmixed type}, see \cite{cat00}.
\end{enumerate}

\nopagebreak
Under these assumption we have. \nopagebreak
\begin{prop}~\cite{cat00}\label{isoinv}
Let $S=(C_1 \times C_2)/G$ be a surface isogenous to a higher product of curves, then $S$ is a minimal surface of general type with the following invariants:
\begin{equation}\label{eq.chi.isot.fib}
\chi(S)=\frac{(g(C_1)-1)(g(C_2)-1)}{|G|},
\quad
e(S)=4 \chi(S),
\quad
K^2_S=8 \chi(S).
\end{equation}
The irregularity of these surfaces is computed by
\begin{equation}\label{eq_irregIsoToProd}
 q(S)=g(C_1/G)+g(C_2/G).
\end{equation}
Moreover the fundamental group $\pi_1(S)$ fits in the following short exact sequence of groups
\begin{equation}\label{eq_fundGroupS}
1 \longrightarrow \pi_1(C_1) \times \pi_1(C_2) \longrightarrow \pi_1(S) \longrightarrow G \longrightarrow 1.
\end{equation}
\end{prop}
Among the nice features of surfaces isogenous to a product, one is that they can be obtained in a pure algebraic way. Let us briefly recall how.  
\begin{defin}\label{genrvect} Let $G$ be a finite group and
let
\[  0 \leq g', \ \ \ \ \ 2 \leq m_1 \leq \dots \leq m_r
\]
be integers. A \emph{system of generators} for $G$ of \emph{type}
$\tau:=(g'\mid m_1,...,m_r)$ is a $(2g'+r)-$tuple of elements of $G$:
\[ \mathcal{V}=(a_1,b_1,\dots,a_{g'},b_{g'},c_1,\dots,c_r)
\]
such that the following conditions are satisfied:
\begin{enumerate}
    \item $\langle a_1,b_1,\dots,a_{g'},b_{g'},c_1,\dots,c_r \rangle \cong G$.
    \item $ord(c_i) = m_i$ for all $1 \leq i \leq r$, denoting by $ord(c)$ the order of $c$.
    \item  $c_1 \cdot \ldots \cdot c_r \cdot \prod^{g'}_{i=1}[a_i,b_i]=1$.
\end{enumerate}
If such a $\mathcal{V}$ exists then $G$ is called \emph{$(g'\mid
m_1, \dots ,m_r)-$generated}.

Moreover, we call the $r$-tuple $(c_1,\dots,c_r)$ the \emph{spherical part} of $\mathcal{V}$ and if $g'=0$ a system of generators is simply said to be \emph{spherical}. 
\end{defin}

We shall also use the notation , for example, $(g' \mid 2^4,3^2)$ to indicate
the tuple $(g' \mid 2,2,2,2,3,3)$. 

\medskip

We have the following reformulation of the Riemann Existence Theorem.

\begin{prop}\label{Prop_RiemEx} A finite group $G$ acts as a
group of automorphisms of some compact Riemann surface $C$ of
genus $g$ if and only if there exist integers $g' \geq 0$ and $m_r
\geq m_{r-1} \geq \dots \geq m_1\geq 2$ such that $G$ is $(g'\mid
m_1,\dots,m_r)-$generated for some system of generators
$(a_1,b_1,\dots,a_{g'},b_{g'},c_1,\dots,c_r)$, and the following
Riemann-Hurwitz relation holds:
\begin{equation}\label{eq.RiemHurw} 2g-2=| G | (2g'-2 +
\sum^r_{i=1}(1-\frac{1}{m_i})).
\end{equation}
\end{prop}
If this is the case, then $g'$ is the genus of the quotient
Riemann surface $C':=C/G$ and the Galois covering $C \rightarrow C'$ is
branched in $r$ points $p_1,\dots,p_r$ with branching numbers
$m_1,\dots,m_r$ respectively. Moreover if $r=0$ the covering is said
to be \emph{unramified} or \emph{\'{e}tale}.

\begin{defin} Two systems of generators $\mathcal{V}_1:=(a_{1,1},b_{1,1},\dots,a_{1,g'_1},b_{1,g'_1},c_{1,1},\dots,c_{1,r_1})$ and $\mathcal{V}_2:=(a_{2,1},b_{2,1},\dots,a_{2,g'_2},b_{2,g'_2},c_{2,1},\dots,c_{2,r_2})$ of $G$ are said to have \emph{disjoint stabilizers} or simply to be
\emph{disjoint}, \index{Disjoint Vectors} if:
\begin{equation}\label{eq.sigmasetcond} \Sigma(\mathcal{V}_1)
\cap \Sigma(\mathcal{V}_2)= \{ 1 \},
\end{equation}
where $\Sigma(\mathcal{V}_i)$ is the set of elements in $G$ that stabilize a point in $C$,
\[ \Sigma(\mathcal{V}_i):= \bigcup_{h \in G} \bigcup^{\infty}_{j=0} \bigcup^{r_i}_{k=1} h \cdot c^j_{i,k} \cdot
h^{-1}.
\]
\end{defin}
We notice that in  the above definition only the spherical part of the system of generators plays a r\^ole. 

\begin{rem}\label{minimal}
From the above discussion we obtain that the datum of a surface
$S$ isogenous to a higher product of curves of unmixed type together with its minimal realization $S=(C_1 \times
C_2)/G$ is determined by the datum of a finite
group $G$ together with two
disjoint systems of generators $\mathcal{V}_1$ and $\mathcal{V}_2$ (for more details see e.g. \cite{BCG06}).
\end{rem}

\begin{rem}\label{rem_singcong} The condition of being
disjoint ensures that the action of $G$ on the product of the two
curves $C_1 \times C_2$ is free.

Indeed the cyclic groups $\left\langle
c_{1,1}\right\rangle,\dots,\left\langle c_{1,r_1}\right\rangle$ and
their conjugates provide the non-trivial stabilizers for the
action of $G$ on $C_1$, whereas $\left\langle
c_{2,1}\right\rangle,\dots,\left\langle c_{2,r_2}\right\rangle$ and
their conjugates provide the non-trivial stabilizers for the
action of $G$ on $C_2$. The singularities of $(C_1 \times C_2)/G$
arise from the points of $C_1 \times C_2$ with non-trivial
stabilizer, since the action of $G$ on $C_1 \times C_2$ is diagonal,
it follows that the set of all stabilizers
for the action of $G$ on $C_1 \times C_2$ is given by
$\Sigma(\mathcal{V}_1) \cap \Sigma(\mathcal{V}_2)$. 
\end{rem}
\begin{defin} Let $\tau_i:=(g'_i \mid m_{1,i}, \dots , m_{r_i,i})$ for $i=1,2$ be two types.  An \emph{(unmixed) ramification structure} of type $(\tau_1,\tau_2)$  for a finite group $G$, is a
pair $(\mathcal{V}_1,\mathcal{V}_2)$
of disjoint systems of generators of $G$, whose types are
$\tau_i$, and which satisfy:
\begin{equation}\label{eq.Rim.Hur.Condition}
\mathbb{Z} \ni \frac{|G|
(2g'_i-2+\sum^{r_i}_{l=1}(1-\frac{1}{m_{i,l}}))}{2}+1 \geq 2,
\end{equation}
for $i=1,2$.
\end{defin}
\begin{rem} Note that a group $G$ and a ramification structure determine the main numerical
invariants of the surface $S$. Indeed, by \eqref{eq.chi.isot.fib} and~\eqref{eq.RiemHurw} we obtain:
\begin{equation}\label{eq.pginfty}
4\chi(S)=|G|\cdot\left({2g'_1-2+\sum^{r_1}_{k=1}(1-\frac{1}{m_{1,k}})}\right)
\cdot\left({2g'_2-2+\sum^{r_2}_{k=1}(1-\frac{1}{m_{2,k}})}\right)=: 4\chi (|G|,(\tau_1,\tau_2)).
\end{equation}
\end{rem}

The most important property of surfaces isogenous to a product is their weak rigidity property.
\begin{theo}~\cite[Theorem 3.3, Weak Rigidity Theorem]{cat04}
Let $S=(C_1 \times C_2)/G$ be a surface isogenous to a higher
product of curves. Then every surface with the same
\begin{itemize}
\item topological Euler number and
\item fundamental group
\end{itemize}
is diffeomorphic to $S$. The corresponding  moduli space
$\mathcal{M}^{top}(S) = \mathcal{M}^{\it diff}(S)$ of surfaces
(orientedly) homeomorphic (resp. diffeomorphic) to $S$ is either
irreducible and connected or consists of two irreducible connected
components exchanged by complex conjugation.
\end{theo}

\begin{rem}
Thanks to the Weak Rigidity Theorem, we have  that the moduli space
of surfaces isogenous to a product of curves with fixed invariants
--- a finite group $G$ and a type $(\tau_1,\tau_2)$ --- consists of a finite number of irreducible connected
components of $\mathcal{M}$. More precisely, let $S$ be a surface
isogenous to a product of curves of unmixed type with group $G$
and a pair of disjoint systems of generators of type
$(\tau_1,\tau_2)$. By~$\eqref{eq.pginfty}$ we have
$\chi(S)=\chi(|G|,(\tau_1,\tau_2))$, and consequently,
by~\eqref{eq.chi.isot.fib}
$K^2_S=K^2(|G|,(\tau_1,\tau_2))=8\chi(S)$, and
$e(S)=e(|G|,(\tau_1,\tau_2))=4\chi(S)$. Moreover, recall that the fundamental group of $S$ fits into the exact sequence \eqref{eq_fundGroupS} and
the subgroup $\pi_1(C_1) \times \pi_1(C_2)$  of $\pi_1(S)$ is unique, see \cite{cat00}. 
\end{rem}

Let us chose a pair $(\tau_1,\tau_2)$ of types.
Denote by $\mathcal{M}_{(G,(\tau_1,\tau_2))}$ the moduli space 
of isomorphism classes of surfaces isogenous to a product,
which have a minimal realization $(C_1 \times C_2)/G$ that
is given by a ramifications structure of type $(\tau_1,\tau_2)$
for the finite group $G$. This moduli space is obviously a
subset of the moduli space
$\mathcal{M}_{K^2(|G|,(\tau_1,\tau_2)),\chi (|G|,(\tau_1,\tau_2))}$.
With $y:=K^2(n,(\tau_1,\tau_2))$ and  
$x:=\chi (n,(\tau_1,\tau_2))$ we get:

\begin{lem}
\label{fund}
Given a positive integer $n$ and a pair $(\tau_1,\tau_2)$ of types,
then $\iota(x,y)$ is bounded from below by 
\[
\# \{ G\,,\: \mbox{ $G$ is a group of order $n$
with a ramification structure of type $(\tau_1,\tau_2)$} 
\}\big/\raisebox{-2mm}{iso.}
\]
\end{lem}

\begin{proof}
It remains to prove that non-isomorphic group lead to distinct
irreducible components of the moduli space. Indeed, 
the fundamental groups of the minimal realizations fit into
sequences as given in Prop.\ref{isoinv}\eqref{eq_fundGroupS},
with non-isomorphic quotients. On the
other hand any isomorphism of the fundamental groups descends
to the quotients since the subgroups are preserved thanks to the minimality of the realizations. So our claim follows.
\end{proof}

\subsection{Enumerating $p$-groups}
\begin{prop}\cite{H60,S65}\label{prop: num p grups} If $f(k,p)$ is the number of groups of order $p^k$, $p$ a prime, and if $A=A(k,p)$ is defined by
\begin{equation}
f(k,p)=p^{Ak^3},
\end{equation}
then
\begin{equation}
\frac{2}{27}-\epsilon_k \leq A \leq \frac{2}{15}-\epsilon_k,
\end{equation}
where $\epsilon_k$ is a positive number, depending only on $k$, which tends to $0$ as $k$ tends to $\infty$.
\end{prop}
We are interested in the constructive part of the proof, where a sufficient number of groups is given, all of them nilpotent of Frattini-class $2$, i.e.\ their Frattini subgroups are central and elementary abelian. Such groups are given by the following presentation.
Let $r$ and $s$ be positive integers with $s+r=k$ and $b(i,j)$, $1 \leq i \leq r$, $1 \leq j \leq s$, and $c(i,i^{\prime},j)$, $1\leq i < i^{\prime} \leq r$, $1 \leq j \leq s$, be integers between $0$ and $p-1$. Then the relations
\[
 \begin{array}{lll} 
(1) & [g_i,g_{i^{\prime}}]=h^{c(i,i^{\prime},1)}_1 \cdot \ldots \cdot h^{c(i,i^{\prime},s)}_s,  & 1\leq i < i^{\prime} \leq r, \\
 
(2) &  [g_{i},h_{j}]=1, & 1 \leq i \leq r, 1 \leq j \leq s, \\

(3) &  [h_j, h_{j^{\prime}}]=1, & 1 \leq j < j^{\prime} \leq s, \\
 
(4) &  g^p_i=h^{b(i,1)}_1 \cdot \ldots \cdot h^{b(i,s)}_s & 1\leq i \leq r, \\
 
(5) &  h^p_j=1, &  1\leq j \leq s, \\
  \end{array}
\] 
on $g_1, \ldots ,g_r$ and $h_1, \ldots h_s$ define a group of order $p^k$.

\begin{rem} We can make the following two restrictions which won't change the asymptotic number of $p-$groups considered.
\begin{itemize}
\item To prove Proposition \ref{prop: num p grups} is enough to consider groups with $r=2s$.
\item We can change (4) to $g^p_i=1$, and consider only groups which are generated by elements of order $p$. This means that the $b(i,j)$'s are set to zero. This is allowed since their proportion among all choices is negligible compered to the number of $c(i,i^{\prime},j)$'s as $s$ goes to infinity. 
\item We can consider only those groups which are generated by the $g_i$'s only. Indeed, such groups are characterized by the property that the $s$ vectors $c_j$ with entries $c_{i, i^{\prime}, j}$ are linearly independent. So the number of possible choices is 
\[ \prod^{s-1}_{l=0}(p^{(\begin{smallmatrix} r \\ 2 \end{smallmatrix})}-p^l).
\]
Again the deviation to $p^{(\begin{smallmatrix} r \\ 2 \end{smallmatrix})s}$ can be subsumed into $\epsilon_k$.
\end{itemize}
\end{rem}

{\bf Assuptions II:} We assume from now on that:
\begin{enumerate}
\item Let $r=2s$.
\item $p=2$. Nevertheless, what follows can be easily extended to $p >2$;
\item All the groups $G$ have a presentation as above with condition (4) changed into $g^p_i=1$.
\item   All the groups $G$ are generated by the $g_i$.
\end{enumerate} 

As seen in the previous section to give a surface isogenous to a product it is enough to give a finite group $G$ and a ramification structure of $G$. In this section we give one of the two systems of generators of a ramification structure that we keep fixed in the next two sections. We will complete the ramification structure of $G$ with a second system of generators, once in order to obtain regular surfaces and then to have irregular ones.

We consider the following system of generator of size $2s+2$ for a group of order $2^k=2^{3s}$.
\begin{equation}\label{eq_FirstGenerator}
T_1:=(g_1, \ldots ,g_s,\bar{g}_s,g_{s+1}, \ldots, g_{2s}, \bar{g}_{2s} ),
\end{equation}  

where $\bar{g}_s=(g_1 \cdots g_s)^{-1}$, $\bar{g}_{2s}=(g_{s+1} \cdots g_{2s})^{-1}$.
By construction we have $<T_1> \cong G$, $g_1 \cdot \ldots \cdot g_s \cdot \bar{g}_s \cdot g_{s+1} \cdot \ldots \cdot g_{2s} \cdot \bar{g}_{2s}=1_G$, and the orders of each element is $2$.    
This gives an action of a group $G$ of order $2^k$ on some curve $C_1$ of genus 
\begin{equation}\label{eq genusC1}
 g(C_1):=2^{3s-1}\Big(-2+\sum^{2s+2}_{l=1}(1-\frac{1}{2})\Big)+1=2^{3s-1}(s-1)+1.
\end{equation}
Moreover, $C_1 \rightarrow C_1/G \cong \PP^1$ branches in $2s+2$ points. In terms of Euler numbers we have:
\[ e(C_1)=2^{3s}\Big(e(\PP^1)-\sum^{2s+2}_{l=1}(1-\frac{1}{2})\Big)=2^{3s}(1-s).
\]


Now we give a criterion to see if two systems of generators are disjoint. Let denote by $H(G) \lhd G$ the subgroup of $G$ generated by the $h_j$'s and $\Phi\colon G \to G/H(G)$.
\begin{lem}\label{lem_imageGen} Let $\mathcal{T}_1$ and $\mathcal{T}_2$ be the two spherical parts of two systems generators $\mathcal{V}_1$ and $\mathcal{V}_2$ of $G$. Moreover, let $B_i=\{\Phi(x)| x\in \mathcal{T}_i, x \notin \langle h_j \rangle\}$ and  $B^{\prime}_i=\{x| x\in \mathcal{T}_i, x \in \langle h_j \rangle\}$. 
If $B_1 \cap B_2 = B^{\prime}_1 \cap B^{\prime}_2 = \emptyset$  then $\Sigma(\mathcal{V}_1) \cap \Sigma(\mathcal{V}_2)=1_G$. 
\end{lem} 
\begin{proof}
Since the order of every element is $2$, it is enough to prove that 
\[ A_1 \cup A_2:= \big( \bigcup_{t \in G} \, \bigcup_{x_1 \in \mathcal{T}_1} tx_{1}t^{-1} \big) \cap \big(\bigcup_{t \in G} \, \bigcup_{x_2 \in \mathcal{T}_2} tx_{2}t^{-1}\big) = \emptyset.
\]
Since the kernel of $\Phi$ is $H(G)$ and the image is abelian, the images of the two sets $A_1$ and $A_2$ are exactly $B_1$ and $B_2$. By hypothesis $B_1 \cap B_2=\emptyset$, so 
\[\Sigma(\mathcal{V}_1) \cap \Sigma(\mathcal{V}_2)=\big(\Sigma(\mathcal{V}_1) \cap H\big) \cap \big(\Sigma(\mathcal{V}_2) \cap H\big)= \big(B^{\prime}_1\cup 1_G \big)\cap \big(B^{\prime}_2 \cup 1_G\big)=1_G.
\]
\end{proof}
\section{Regular Surfaces}
As said before, we give the second system of generators of a ramification structure for a group $G$ as above which yields a regular surface isogenous to a product. Let 
\begin{equation}
T_2:=(g_1g_2, g_2g_3, \ldots ,g_{s-1}g_s,g_sg_2g_3,(g_1g_2g_3)^{-1},g_{s+1}g_{s+2}, \ldots, g_{2s}g_{s+2}g_{s+3}, (g_{s+1}g_{s+2}g_{s+3})^{-1}),
\end{equation}
One can see that $<T_2> \cong G$ and by construction the product of the elements in $T_2$ is $1_G$. This yields a second $G$-Galois cover of $\PP^1$ ramified in $2s+2$ points. And again the genus of the curve is

\begin{equation}\label{eq genusC2}
g(C_2)=2^{3s-1}(s-1)+1.
\end{equation}

Moreover, the set $B_2:=\{ \Phi(x) | x \in T_2\}$ is disjoint from $B_1:=\{ \Phi(x) | x \in T_1\}$ so by the Lemma \ref{lem_imageGen} the pair $(T_1,T_2)$ is a ramification structure for $G$. The associated surface isogenous to a product $S$ has irregularity $q(S)=0$. 

This ramification structure can be given to any $2$-group as above. 

\begin{theo}\label{Thm regular Main} Let $h$ be number of connected components of regular surfaces isogenous to a product of curves admitting a group of order $2^{3s}$ and ramification structure of type $((0 | 2^{2s+2}),(0 | 2^{2s+2}))$, as above. Then for $s \rightarrow \infty$ we have  
\begin{equation}\label{eq hreg} 
h \geq 2^{Bs^3},
\end{equation}
where $2-\epsilon^{\prime}_s \leq B \leq \frac{18}{5} - \epsilon^{\prime}_s$, and $\displaystyle \lim_{s \rightarrow \infty}\epsilon^{\prime} = 0$.
All these surfaces are regular, i.e. $q(S)=0$ and 
\begin{equation}\label{eq_chiRegS}
\chi(S)=2^{3s-2}(s-1)^2.
\end{equation}
\end{theo} 
\begin{proof}
For fixed group order $2^{3s}$ the number of groups with this order is $2^{27As^3}$ by Proposition \ref{prop: num p grups}. For each of these groups we found a ramification structure $(T_1,T_2)$. By Lemma \ref{fund}, the number of surfaces isogenous to a product associated to those data is at least $2^{Bs^3}$, where $B$ is as in the claim. 
By \eqref{eq genusC1} and \eqref{eq genusC2} the holomorphic Euler characteristic of $S$ is 
\[ \chi(S)=\frac{(g(C_1)-1)(g(C_2)-1)}{|G|}=2^{3s-2}(s-1)^2.
\]

\end{proof}
\section{Irregular Surfaces}
Let us consider the irregular case i.e. $q(S)=q>0$. Let $h:=[g_{s},g_{2s}]\cdot [g_{s-1},g_{2s-1}] \cdots [g_{s-q+1}, g_{2s-q+1}]$ and recall that the commutators are in the center of the group $G$. Moreover, let 
\[
T^{\prime}_2:=\{g_1g_2, g_2g_3, \ldots ,g_{s-q}g_{s-q+1},g_{s-q+1}g_1,g_{s+1}g_{s+2}, \ldots ,\]
\[ g_{2s-q}g_{2s-q+1},
g_{2s-q+1}g_{s+1}h\}
\]

the spherical part of the generating vector
\[\mathcal{V}_2:=\{T^{\prime}_2,g_{s},g_{2s},g_{s-1},g_{2s-1}, \ldots g_{s-q+1}, g_{2s-q+1}\}
\]

It holds $<\mathcal{V}_2> \cong G$ and and by construction the product of the elements in $\mathcal{V}_2$ is $1_G$.
Moreover, by Proposition \ref{Prop_RiemEx} this yields an action of a group $G$ of order $2^{3s}$ on some curve $C_2$ of genus 
\[ g(C_2):=2^{3s-1}(2q-2+(s-q+1))+1=2^{3s-1}(s+q-1)+1
\]
with $g(C_2/G)=q$.

By construction and by Lemma \ref{lem_imageGen} $T_1$ and $T^{\prime}_2$ are disjoint and so $(T_1,\mathcal{V}_2)$ is a ramification structure for the groups as above. These data give us a surface isogenous to a product $S:=(C_1 \times C_2)/G$, where $G$ is a group of order $2^{3s}$, which satisfies the assumption above. The number of these surfaces is at least $2^{27As^3}$. Then the following theorem is proven in analogy to Theorem \ref{Thm regular Main}.
\begin{theo} Let $h$ be number of connected components of the moduli space of surfaces of general type isogenous to a product of curve admitting a group of order $2^{3s}$ and ramification structure of type $((0 | 2^{2s+2}),(q | 2^{2s-2q+2})$ as above. Then we have  
\[ h \geq 2^{Bs^3},
\]
where $2-\epsilon^{\prime}_s \leq B \leq \frac{18}{5} - \epsilon^{\prime}_s$, and
$\displaystyle \lim_{s \rightarrow \infty}\epsilon^{\prime} = 0$. All these surfaces are irregular and have irregularity $q(S)=q$. Finally,
The holomorphic Euler characteristic of $S$ is 
\[ \chi(S):=\frac{(g(C_1)-1)(g(C_2)-1)}{|G|}=2^{3s-2}(s-1)(s+q-1).
\]
\end{theo}
\section{Proof of Theorem \ref{thm main}}

We give the proof for regular surfaces only, the irregular case being analogous.

We start with \eqref{eq_chiRegS} and we write
\begin{equation*}
x_s=2^{3s-2}(s-1)^2.
\end{equation*}
Since $x_s$ is strictly monotonically increasing with $s$, there is a well defined inverse function $s=s(x)$. From $\log_2 x_s = 3s-2+2\log_2(s-1)$ we decuce 
\begin{equation*}
s(x)=\frac{1}{3}(1+\eta_{x_s})\log_2 x_s
\end{equation*}
with $\eta_{x_s} \rightarrow 0$ for $x_s \rightarrow \infty$. Substituting $s(x)$ into the inequality \eqref{eq hreg} we get
\begin{equation*}
h \geq 2^{\frac{B}{27}(\log_2 x_s)^3(1+\eta_{x_s})^3}.
\end{equation*} 
For $x_s$ large enough this is bounded from below by $2^{\frac{2}{9}(\ln x_s)^3}$, thanks to $0 < 27(\ln 2)^3<9$. 

We use the identity $x^{f(x)}=e^{f(x)\ln x}=2^{f(x)\frac{1}{\ln 2}\ln x}$ to derive 
\begin{equation*}
h \geq x_s^{(\frac{2}{27(\ln 2)^2}(\ln x_s)^2)\frac{1}{\ln 2}\ln x_2}=x_s^{\frac{2}{13}(\ln x_s)^2},
\end{equation*}
thanks to $0 <27(\ln 2)^2<13$. Since $y_s$ is a constant multiple of $x_s$ the asymptotic is the same. This concludes the proof of the theorem.



\end{document}